\documentclass[11pt, reqno]{amsart}

\usepackage{fullpage}
\usepackage{amsmath}
\usepackage{amsthm}
\usepackage{amsfonts}
\usepackage{amssymb}

\newtheorem{lemma}{Lemma}[section]
\newtheorem{proposition}[lemma]{Proposition}
\newtheorem{theorem}[lemma]{Theorem}
\newtheorem{corollary}[lemma]{Corollary}
\theoremstyle{remark}

\theoremstyle{definition}

\numberwithin{equation}{section}

\newcommand{\mb}{\mathbf}
\newcommand{\mc}{\mathcal}

\newcommand{\R}{{\mathbb R}}
\newcommand{\N}{{\mathbb N}}

\newcommand{\C}{{\mathbb C}}
\newcommand{\rst}[1]{\ensuremath{{\mathbin |}%
\raise-.5ex\hbox{$#1$}}} 

\newcommand{\bs}[1]{\boldsymbol{#1}}

\title{Stable self-similar blow up for energy subcritical wave equations}

\author{Roland Donninger}
\address{\'Ecole Polytechnique F\'ed\'erale de Lausanne, 
Department of Mathematics, Station 8, CH-1015 Lausanne, Switzerland}
\email{roland.donninger@epfl.ch}

\author{Birgit Sch\"orkhuber}
\address{Vienna University of Technology, Institute of Analysis and Scientific Computing,
Wiedner Hauptstra\ss e 8,  A-1040 Vienna, Austria}
\email{birgit.schoerkhuber@tuwien.ac.at}
\thanks{The second author acknowledges partial support from the Austrian Science Fund
(FWF), grants P22108, P24304, and I395; the Austrian-French Project of
the Austrian Exchange Service (\"OAD); and the Innovative Ideas
Program of Vienna University of Technology.}

\begin{document}

\maketitle

\begin{abstract}
We consider the semilinear wave equation
\[ \partial_t^2 \psi-\Delta \psi=|\psi|^{p-1}\psi \]
for $1<p\leq 3$ with radial data in $\R^{3}$.
This equation admits an explicit spatially homogeneous blow up solution $\psi^T$ given by
$$ \psi^T(t,x)=\kappa_p (T-t)^{-\frac{2}{p-1}} $$
where $T>0$ and $\kappa_p$ is a $p$-dependent constant.
We prove that the blow up described by $\psi^T$ is stable against small perturbations
in the energy topology.
This complements previous results by Merle and Zaag.
\end{abstract}

\section{Introduction}

\noindent We study the Cauchy problem
\begin{equation} \left \{
\begin{array}{l}
\label{eq_pwave}
\partial_{t}^2\psi-\Delta \psi=|\psi|^{p-1}\psi \\
\psi[0]=(f,g) \end{array} \right. 
\end{equation}
for $\psi: [0,\infty) \times \R^3 \to \R$ and $1<p \leq 3$.
The associated energy reads
\[ E(\psi(t,\cdot),\psi_t(t,\cdot))=\frac12 \int_{\R^3} \left (|\psi_t(t,x)|^2
+|\nabla_x \psi(t,x)|^2 \right )dx-\frac{1}{p+1}\int_{\R^3}|\psi(t,x)|^{p+1}dx.
 \]
Furthermore, Eq.~\eqref{eq_pwave} is invariant under the scaling transformation
\[\psi(t,x) \mapsto \psi_\lambda(t,x):=\lambda^{-\frac{2}{p-1}}\psi(t/\lambda,x/\lambda) \] 
for $\lambda>0$ and
the energy scales according to
\[ E(\psi_\lambda(t,\cdot),\partial_t \psi_\lambda(t,\cdot))=\lambda^{-\frac{5-p}{p-1}}
E(\psi(t/\lambda,\cdot),\psi_t(t/\lambda,\cdot)). \]
Thus, Eq.~\eqref{eq_pwave} is energy subcritical if $1<p<5$, critical if $p=5$ and supercritical
if $p>5$.
Note, however, that the usual blow up heuristics (energy conservation prevents the solution from
shrinking to ever smaller scales in the subcritical case) do not apply here since the energy is
not positive definite.

Local well-posedness of Eq.~\eqref{eq_pwave} in $H^1(\R^3)\times L^2(\R^3)$ follows by standard arguments (use
Duhamel's formula and Sobolev embedding to set up
a contraction scheme).
However, Eq.~\eqref{eq_pwave} is not globally well-posed and it is well-known that initial data
with negative energy lead to singularity formation in finite time \cite{Lev74}.
More explicitly, one may look for self-similar solutions
which are by definition invariant under the natural scaling.
A particularly simple example of this type is obtained by neglecting the Laplacian altogether and solving
the remaining ODE in $t$
which yields the solution
\[ \psi^T(t,x)=\kappa_0^{\frac{1}{p-1}}(T-t)^{-\frac{2}{p-1}} \] 
where $T>0$ and $\kappa_0=\frac{2(p+1)}{(p-1)^2}$.
We refer to $\psi^T$ as the \emph{fundamental self-similar solution}. 
Although $\psi^T$ is homogeneous in space, by finite speed of propagation it can still be used to construct
compactly supported smooth initial data that lead
to blow up at time $t=T$.
In fact, one expects that Eq.~\eqref{eq_pwave} admits many self-similar blow up solutions, even in
the radial case. 
At least for $p=3$ this was proved by Bizo\'n et.~al.~\cite{BBMW10} (there are similar
results for $p\geq 7$ \cite{BMW07}).
We remark that the situation for the corresponding problem in one space dimension is fundamentally 
different since in this case there exists a unique (up to symmetries) self-similar blow up
solution. 

In order to understand the dynamics of Eq.~\eqref{eq_pwave} it is important to 
analyse possible blow up scenarios.
In two remarkable papers \cite{MerZaa03}, \cite{MerZaa05} Merle and Zaag proved
that \emph{any} blow up solution of Eq.~\eqref{eq_pwave} 
blows up at the rate $(T-t)^{-\frac{2}{p-1}}$.
However, the precise shape of the blow up profile depends on the data.
Numerical work by Bizo\'n, Chmaj and Tabor \cite{BCT04} suggests that the blow up described by $\psi^T$ is
the ``generic'' one.
To be more precise, they observe in their simulations that the future development of 
``generically'' chosen radial blow up data
converges to $\psi^T$ locally near the center of spherical symmetry.
In the present paper we analyse the stability of $\psi^T$ and obtain the following result, see 
Theorem \ref{Th:Main} below for the precise formulation.

\begin{theorem}[main result, qualitative formulation]
There exists an open set (in the energy topology) of radial data that lead to blow up
via $\psi^T$. In other words, the blow up described by $\psi^T$ is stable.
\end{theorem}

Here, ``energy topology'' refers to the topology generated by the energy of the \emph{free} wave
equation. In fact, we use a slight modification thereof which has better local behavior, see below.
Before going into more details, let us briefly comment on possible extensions of our result.
First of all, the restriction to radial data is only technical and can be quite 
easily removed.
However, since the nonradial case is not expected to reveal any new interesting phenomena,
we stick to radial data so as not to obscure the exposition by too many technicalities.
A more important aspect concerns the admissible values of $p$.
It should be possible to extend our result to the full subcritical regime
$1<p<5$ with some modifications.
For instance, it is necessary to require slightly more regularity (e.g.~of Strichartz type)
if $3<p<5$ in order to control the nonlinearity.
This is reminiscent of the standard well-posedness theory for semilinear wave equations and we
will pursue this matter elsewhere.
We even claim that the stability of $\psi^T$ holds true in the full supercritical
regime $p\geq 5$ albeit in a topology significantly stronger than the energy, cf.~\cite{wavemapslin}, 
\cite{wavemapsnonlin} for analogous results on supercritical wave maps.

In a certain sense our work is complementary to the results of Merle and Zaag \cite{MerZaa05}.
In their approach they consider \emph{any} blow up solution whereas our method is perturbative near
$\psi^T$.
Consequently, the strength of their result lies in its generality: it is a statement about
all possible blow up scenarios.
However, it is not true that every blow up solution converges to $\psi^T$ since 
there exist many self-similar profiles.
We remark that the situation in the one dimensional case is entirely different due
to the uniqueness of the self-similar solution.
As a consequence, the Merle-Zaag method even yields profile convergence for the 
problem in $\R^{1+1}$ \cite{MerZaa07}.
Furthermore, a number of beautiful results emerged from this approach, see, e.g.~\cite{MerZaa08}, 
\cite{MerZaa08a}, \cite{MerZaa10}, \cite{CotZaa11}.

Our method of proof follows essentially our work on the wave maps problem \cite{wavemapslin}, 
\cite{wavemapsnonlin}.
However, there are important differences. 
First of all, the problem at hand is energy subcritical and this allows us to work in
the energy topology whereas for the wave maps problem we had to require more regularity.
Furthermore, the present result is completely rigorous and there is no need for any 
numerical input as was the case in \cite{wavemapsnonlin}.
This is because the nonself-adjoint spectral problem associated to the linear stability of
$\psi^T$ can be solved explicitly in terms of hypergeometric functions.
Finally, we have improved the method in order to avoid one additional degree of differentiability
which was still necessary in \cite{wavemapsnonlin}.
Consequently, the motivation for this paper is in fact twofold:
First, our result is intended to complement the work of Merle and Zaag in order to obtain a fairly 
satisfactory description of the blow up behavior of Eq.~\eqref{eq_pwave}. 
Second, we wanted to demonstrate the wide applicability of our methods developed in 
\cite{wavemapslin} and \cite{wavemapsnonlin}.

Our approach is functional analytic. 
We first introduce new coordinates adapted to self-similarity 
and linearize the equation around the solution $\psi^T$.
The resulting linear problem involves a highly nonself-adjoint operator (due to the change
of coordinates) and it is therefore studied by semigroup methods.
This yields the linear stability of $\psi^T$.
In fact, this stability is modulo the time translation symmetry of the problem which manifests itself
in the form of an unstable mode of the linearized operator.
We remove this instability by a Riesz projection.
The nonlinear stability is then proved by a fixed point argument where we have to take
into account the instability caused by the symmetry.
In order to deal with this problem we employ an infinite-dimensional version of the Lyapunov-Perron method
from dynamical systems theory: we force nonlinear stability by
modifying the data.
In a last step we undo this modification by shifting the blow up time.

\subsection{Formulation of the Cauchy problem and the main result}
We restrict ourselves to radial solutions and study the initial value problem in the backward 
lightcone $\mc{C}_T$ of the blow up point $(T,0)$ which is defined by
$$ \mc{C}_T:=\{(t,r): t \in (0,T), r \in [0,T-t]\}. $$
More precisely, we consider
\begin{equation}
\label{eq:maincauchy}
\left \{ \begin{array}{l}
\psi_{tt}(t,r)-\psi_{rr}(t,r)-\frac{2}{r}\psi_r(t,r)-|\psi(t,r)|^{p-1}\psi(t,r)=0 \mbox{ for } (t,r) \in \mc{C}_T \\
\psi(0,r)=f(r), \psi_t(0,r)=g(r) \mbox{ for }r \in [0,T]
         \end{array} \right .
\end{equation}
with given initial data $(f,g)$. We are interested in the stability of $\psi^T$ under small perturbations $\varphi$. 
Thus, we insert the ansatz $\psi=\psi^T+\varphi$ into Eq.~ (\ref{eq:maincauchy}) and expand the nonlinearity according to
$$|\psi^T+\varphi|^{p-1}(\psi^T+\varphi)= |\psi^T|^{p-1}\psi^T + p|\psi^T|^{p-1} \varphi + N_T(\varphi),$$ 
where $N_T$ denotes the nonlinear remainder. Since $\psi^T(t,r) > 0$ for all $(t,r) \in \mc{C}_T $ we obtain

\begin{equation}
\label{eq:pwavecauchy}
\left \{ \begin{array}{l}
\varphi_{tt}-\varphi_{rr}-\frac{2}{r}\varphi_r-p(\psi^T)^{p-1} \varphi - N_T(\varphi)=0 \mbox{ in } \mc{C}_T \\
\varphi(0,r)=f(r)-\psi^T(0,r), \varphi_t(0,r)=g(r)-\psi^T_t(0,r) \mbox{ for }r \in [0,T].
\end{array} \right .
\end{equation}

\subsection{Energy norm}

We want to study the Cauchy problem in a backward lightcone and, since our approach is perturbative, we need a $\textit{local energy norm}$ derived from the conserved energy associated to the \textit{free equation} 
\begin{equation}\label{Eq:Free}
\varphi_{tt}-\varphi_{rr}-\frac{2}{r}\varphi_r=0,
\end{equation}
which is given by
$$\int_0^{\infty} r^2 [\varphi_t(t,r)^2 + \varphi_r(t,r)^2]dr.$$
However, this does not define such a norm due to the lack of a boundary condition for $\varphi$ at $r=0$. 
By integration by parts (and assuming sufficient decay at infinity) 
it can be easily seen that the above expression for the energy is equivalent to 
\begin{equation}\label{Def:FreeEnergy}
E(\varphi) = \int_0^{\infty} r^2\varphi_t(t,r)^2 +[r\varphi_r(t,r)+ \varphi(t,r)]^2 dr.
\end{equation}
Another way of motivating this is to define $\tilde \varphi := r \varphi$, such that Eq.~ (\ref{Eq:Free}) transforms to the $1+1$ wave equation
$$\tilde \varphi_{tt} - \tilde \varphi_{rr} = 0$$
with conserved energy 
$$\int_0^{\infty} \tilde \varphi_t(t,r)^2 + \tilde \varphi_r(t,r)^2 dr.$$
Writing this expression in terms of the original field yields (\ref{Def:FreeEnergy}). For $(f,g) \in C^1[0,R] \times C[0,R]$, $R >0$ we define
\begin{equation}\label{Def:LocalEnergy}
\|(f,g)\|_{\mc E(R)}^2 := \int_0^R |r f'(r) + f(r)|^2 dr + \int_0^R r^2 |g(r)|^2 dr. 
\end{equation}
$r f'(r) + f(r)=0$ implies $f(r)=\frac{c}{r}$ and the requirement $f \in C^1[0,R]$ yields $c=0$ such that $\mc \| \cdot \|_{\mc E(R)}$ defines a norm on $C^1[0,R] \times C[0,R]$.
We consider the expression (\ref{Def:LocalEnergy}) in the backward lightcone of the blow up point $(T,0)$ and insert the fundamental self--similar solution to obtain
\begin{equation}\label{energy_blowup_groundstate}
\|(\psi^T(t,\cdot),\psi_t^T(t,\cdot))\|_{\mc E(T-t)} = C_p (T-t)^{-\frac{5-p}{2(p-1)}} 
\end{equation}
where $C_p > 0$ denotes a $p$-dependent constant. Evidently, as $t \to T-$,
this quantity blows up in the energy subcritical case, i.e., for $1<p<5$.

\begin{theorem}[main result, quantitative version]\label{Th:Main}
Fix $1< p\leq 3$ and $\varepsilon>0$. Let $(f,g)$ be radial initial data with
$$\|(f,g) - (\psi^1(0,\cdot), \psi_t^1(0,\cdot))\|_{\mc E(\frac{3}{2})}$$
sufficiently small. Then there exists a $T>0$ close to $1$ such that the Cauchy problem
\begin{equation}
\left \{ \begin{array}{l}
\partial_t^2 \psi-\Delta \psi = |\psi|^{p-1} \psi  \\ 
\psi[0]=(f,g) 
\end{array} \right .
\end{equation}
has a unique radial solution $\psi:  \mc{C}_T \to \R$ which
satisfies
$$ (T-t)^{\frac{5-p}{2(p-1)}}\|(\psi(t,\cdot),\psi_t(t,\cdot))-(\psi^T(t,\cdot),\psi^T_t(t,\cdot))\|_{\mc E(T-t)} 
\leq C_{\varepsilon} (T-t)^{|\omega_p|-\varepsilon}$$
for all $t \in [0,T)$
where $\omega_p := \max \left \{-1,\tfrac{1}{2} - \tfrac{2}{p-1} \right\}$ 
and $C_{\varepsilon} > 0$ is a constant 
which depends on $\varepsilon$.
\end{theorem}

\section{Transformation to a first order system and similarity coordinates}

\subsection{First--order formulation}
By setting $\tilde{\varphi}(t,r):=r\varphi(t,r)$, Eq.~\eqref{eq:pwavecauchy} transforms into
\begin{equation}
\label{eq_pwavetilde}
\left \{ \begin{array}{l}
\tilde{\varphi}_{tt}-\tilde{\varphi}_{rr}-p(\psi^T)^{p-1} \tilde{\varphi} - rN_T(\frac{\tilde{\varphi}}{r})=0 \mbox{ in } \mc{C}_T \\
\tilde{\varphi}(0,r)=r[f(r)-\psi^T(0,r)], \tilde{\varphi}_t(0,r)=r[g(r)-\psi^T_t(0,r)] \mbox{ for }r \in [0,T]
         \end{array} \right .
\end{equation}
together with the boundary condition $\tilde{\varphi}(t,0)=0$ for all $t$.
We rewrite Eq.~ (\ref{eq_pwavetilde}) as a first--order system by introducing new variables
$$\varphi_1=(T-t)^{\frac{2}{p-1}} \tilde{\varphi}_t, \quad \varphi_2=(T-t)^{\frac{2}{p-1}} \tilde{\varphi}_r .$$
Thanks to the boundary condition we have $\tilde{\varphi}(t,r) = (T-t)^{-\frac{2}{p-1}} \int_0^r \varphi_2(t,r')dr'$.
Note that the nonlinearity transforms according to
$$N_T \left((T-t)^{-\frac{2}{p-1}}  r^{-1}\smallint \varphi_2 \right) = (T-t)^{-\frac{2p}{p-1}}  N(r^{-1} \smallint \varphi_2)$$
where $\int \varphi_2$ is shorthand for $\int_0^r \varphi_2(t,r')dr'$ and
\begin{equation}
\label{eq:N}
N(x)=|\kappa_0^{\frac{1}{p-1}}+x |^{p-1}  (\kappa_0^{\frac{1}{p-1}}+x )-\kappa_0^{\frac{p}{p-1}} - p \kappa_0 x.
\end{equation}
Thus, Eq.~ (\ref{eq_pwavetilde}) reads
\begin{equation*}
\left \{ \begin{array}{l}
\left .\begin{array}{l}
\partial_t \varphi_1= \partial_r \varphi_2 - \frac{2}{p-1} (T-t)^{-1}  \varphi_1 + p \kappa_0  (T-t)^{-2}\int  \varphi_2 + r (T-t)^{-2} N(r^{-1}\int  \varphi_2)
 \\
\partial_t \varphi_2= \partial_r \varphi_1 - \frac{2}{p-1} (T-t)^{-1}  \varphi_2 \end{array} \right \}
\mbox{ in } \mc{C}_T \\
\left. \begin{array}{l}
\varphi_1(0,r)=T^{\frac{2}{p-1}} r\left [g(r)-\psi^T_t(0,r) \right ] \\
\varphi_2(0,r)= T^{\frac{2}{p-1}} [rf'(r)+f(r) -\psi^T(0,r)]
\end{array} \right \}
\mbox{ for } r \in [0,T]
\end{array} \right .
\end{equation*} 

\subsection{Similarity coordinates}
We transform the system to similarity coordinates $(\tau,\rho)$, which are given by
$$\tau=-\log(T-t), \quad \rho=\frac{r}{T-t}.$$ 
The cone $\mc{C}_T$ gets mapped to the infinite cylinder $ \mc{Z}_T:=\{(\tau,\rho): \tau>-\log T, \rho \in [0,1]\}$ and by
setting
$$\phi_j(\tau,\rho) := \varphi_j(T-e^{-\tau},e^{-\tau} \rho)$$
for $j=1,2$ we obtain

\begin{equation}\label{eq:nonlinear_firstorder_css}
\left \{ \begin{array}{l}
\left. \begin{array}{l}
\partial_\tau \phi_1=-\rho \partial_\rho \phi_1 + \partial_\rho  \phi_2 - \frac{2}{p-1} \phi_1 + p\kappa_0  \int_0^\rho \phi_2(\tau,s)ds + \rho N(\rho^{-1} \int_0^\rho \phi_2(\tau,s)ds)
 \\
\partial_\tau \phi_2= -\rho \partial_\rho \phi_2 + \partial_\rho  \phi_1 - \frac{2}{p-1} \phi_2 \end{array} \right \}
\mbox{ in }\mc{Z}_T \\
\left. \begin{array}{l}
\phi_1(-\log T,\rho)= \rho \left [T^{\frac{p+1}{p-1}}g(T\rho)-\tfrac{2}{p-1}\kappa_0^{\frac{1}{p-1}} \right ]\\
\phi_2(-\log T,\rho)=T^{\frac{2}{p-1}}\left [ T \rho f'(T\rho)+f(T\rho) \right ] -  \kappa_0^{\frac{1}{p-1}}
\end{array} \right \}
\mbox{ for } \rho \in [0,1].
\end{array} \right .
\end{equation} 

Tracing back the above transformations, the original field $\psi$ as well as its time-derivative can be reconstructed according to
\begin{equation}\label{Eq:ReconstructField}
\begin{aligned}
\psi(t,r) &= \psi^T(t,r) + (T-t)^{-\frac{2}{p-1}} r^{-1} \int_0^r \phi_2(-\log(T-t),\tfrac{r'}{T-t}) dr', \\
\psi_t(t,r)&= \psi_t^T(t,r) + (T-t)^{-\frac{2}{p-1}} r^{-1} \phi_1(-\log(T-t),\tfrac{r}{T-t}).
\end{aligned}
\end{equation}

Note that most of the expressions we are going to define below will depend on $p$. However, for the sake of readability we will not indicate this dependence explicitly, but consider $p$ to be fixed, where we restrict ourselves to $1<p\leq 3$.
\section{Linear Perturbation Theory}

In this section we consider the linearized problem
\begin{equation}\label{eq:linCSS}
\left \{ \begin{array}{l}
\left. \begin{array}{l}
\partial_\tau \phi_1=-\rho \partial_\rho \phi_1 + \partial_\rho  \phi_2 - \frac{2}{p-1} \phi_1 + p\kappa_0  \int_0^\rho \phi_2(\tau,s)ds 
 \\
\partial_\tau \phi_2= -\rho \partial_\rho \phi_2 + \partial_\rho  \phi_1 - \frac{2}{p-1} \phi_2 \end{array} \right \}
\mbox{ in }\mc{Z}_T \\
\left. \begin{array}{l}
\phi_1(-\log T,\rho)= \rho \left [T^{\frac{p+1}{p-1}}g(T\rho)-\tfrac{2}{p-1}\kappa_0^{\frac{1}{p-1}} \right ]\\
\phi_2(-\log T,\rho)=T^{2/(p-1)}\left [T \rho f'(T\rho)+f(T\rho) \right ] -  \kappa_0^{\frac{1}{p-1}}
\end{array} \right \}
\mbox{ for } \rho \in [0,1]
\end{array} \right .
\end{equation} 
which has already been studied in \cite{roland1}. 
Nevertheless, in order to present a consistent picture we summarize known results and supplement 
them by some new aspects which will be important for the nonlinear theory (see Lemma \ref{lem:algebraic_multi}).

\subsection{Well-posedness of the linearized equation}

Let $\mc H := L^2(0,1) \times L^2(0,1)$ with the usual inner product. We define operators $(\tilde L_0, \mc D(\tilde L_0))$ and $L' \in \mathcal{B}(\mathcal{H})$ by
$$\mc D(\tilde L_0):=\{\mb{u} \in C^1[0,1] \times C^1[0,1]: u_1(0)=0\}, $$ 
$$ \tilde L_0\mb{u}(\rho):=\left ( \begin{array}{c}u_2'(\rho)-\rho u_1'(\rho)-\frac{2}{p-1} u_1(\rho)  \\ u_1'(\rho)-\rho u_2'(\rho)-\frac{2}{p-1} u_2(\rho) \end{array}  \right)  $$
and $$ L' \mathbf{u}(\rho):=\left ( \begin{array}{c} p\kappa_0 \int_0^\rho u_2(s)ds 
\\ 0 \end{array} \right )$$
where $\mb u = (u_1,u_2)^T$. It is easy to see that $L'$ is a compact operator, which will play an important role later on.

\begin{lemma}
\label{L0_semigroup}
The operator $\tilde{L}_0$ is closable and its closure $L_0$ generates a
strongly continuous one--parameter semigroup $S_0: [0,\infty) \to \mathcal{B}(\mathcal{H})$
satisfying $\|S_0(\tau)\|\leq e^{\tilde \omega_p\tau}$ for all $\tau \geq 0$ and $\tilde \omega_p:= \frac{1}{2}-\frac{2}{p-1}$.
\end{lemma}
\begin{proof}
The claim is a consequence of the Lumer-Phillips Theorem (see \cite{engel}, p.~83, Theorem 3.15).
Indeed, a simple integration by parts yields the estimate 
$$\mathrm{Re}(\tilde{L}_0 \mathbf{u}|\mathbf{u})\leq
\left(\tfrac{1}{2}-\tfrac{2}{p-1} \right) \|\mathbf{u}\|^2 $$
and $\tfrac{1}{2}-\tfrac{2}{p-1}<0$.
Furthermore, for $\lambda:=1-\frac{2}{p-1} > \tilde \omega_p$ the range of 
$\lambda-\tilde{L}_0$ is dense in $\mathcal{H}$. This follows from the very same calculation as in 
the proof of Lemma 2 in \cite{roland1}. Since $\tilde L_0$ is densely defined, the 
Lumer-Phillips Theorem applies.
\end{proof}

\begin{corollary}
\label{Cor_specL0}
The spectrum of $L_0$ is contained in a left half plane, 
$$ \sigma(L_0)\subset \left \{\lambda \in \mathbb{C}: Re \lambda \leq \tilde \omega_p \right \}, $$
with $\tilde \omega_p= \frac{1}{2}-\frac{2}{p-1}$ 
and the resolvent of $L_0$ satisfies 
$$ \|R_{L_0}(\lambda)\|\leq \frac{1}{Re \lambda-\tilde \omega_p} $$
for all $\lambda \in \mathbb{C}$ with $Re \lambda>\tilde \omega_p$.
\end{corollary}

\begin{proof}
The structure of the spectrum as well as the resolvent estimate follow by standard results of semigroup 
theory (see \cite{engel}, p.~55, Theorem 1.10).
\end{proof}

The next corollary is a consequence of the Bounded Perturbation Theorem (see \cite{engel}, p.~158).

\begin{corollary}[well-posedness of the linearized equation]
\label{cor:gen}
 The operator $L:=L_0+L'$, $\mc D(L):=\mc D(L_0)$ generates a strongly continuous one--parameter semigroup $S: [0,\infty) \to \mc{B}(\mc{H})$ satisfying
 $$ \|S(\tau)\|\leq e^{(\tilde \omega_p+p\kappa_0)\tau} $$
 for all $\tau \geq 0$ and $\tilde \omega_p= \frac{1}{2}-\frac{2}{p-1}$.
 In particular, the Cauchy problem
\begin{equation*}
 \left \{ \begin{array}{l}
\frac{d}{d\tau}\Phi(\tau)=L \Phi(\tau) \mbox{ for }\tau>-\log T \\
\Phi(-\log T)=\mb{u}
          \end{array} \right .
\end{equation*}
has a unique solution given by
$$ \Phi(\tau)=S(\tau+\log T)\mb{u} $$
for $\mb{u} \in \mc{D}(L_0)$ and all $\tau \geq -\log T$.
\end{corollary}

\subsection{Properties of the generator}

We obtain a more explicit characterization of $\mc{D}(L)$ in order to be able to describe
the spectrum of $L$.

\begin{lemma}
\label{Lemma:closure}
Let $\mb u \in \mc D(L)$. Then $\mb u \in C[0,1) \times C[0,1)$ and $u_1(0)=0$. Furthermore, 
for $\mb f \in \mc H$ the equation $(\lambda - L) \mb u = \mb f$ implies 
$$u_1(\rho)= \rho u_2(\rho)+ (\lambda  + \tfrac{2}{p-1} -1  )\int_0^\rho u_2(s)ds -\int_0^\rho f_2(s)ds $$
and 
\begin{equation}
\label{Eq:inhom_eigenval}
\begin{aligned} 
-(1-\rho^2) & u''(\rho)+2\left(\lambda + \tfrac{2}{p-1} \right) \rho u'(\rho) + \left(\left(\lambda+ \tfrac{2}{p-1} \right)\left(\lambda + \tfrac{2}{p-1} - 1\right) - p\kappa_0 \right) u(\rho) \\ & = f_1(\rho) + \rho f_2(\rho) + (\lambda + \tfrac{2}{p-1}) \int_0^\rho f_2(s) ds
\end{aligned}
\end{equation}
in a weak sense, where $u \in H^2_\mathrm{loc}(0,1)\cap C[0,1] \cap C^1[0,1)$ is defined by $u(\rho) := \int_0^\rho u_2(s)ds$.
\end{lemma}

\begin{proof}
Let $\mb u \in \mc D(L) = \mc D(L_0)$. By definition there exists a sequence 
$(\mb u_j) \subset \mc D(\tilde L_0) \subset C^1[0,1] \times C^1[0,1]$ such that 
$\mb u_j \to \mb u$ and $\tilde L_0 \mb u_j \to L_0 \mb u$ in $\mc H$. By combining the expressions for the individual components in an appropriate way we infer that $(1-\rho^2)u'_{1j}$ and $(1-\rho^2)u'_{2j}$ are convergent sequences in $L^2(0,1)$. Thus $u_1,u_2 \in H^1(0,1-\varepsilon) \hookrightarrow C[0,1-\varepsilon]$ for any $\varepsilon \in (0,1)$. This guarantees the boundary condition $u_1(0)=0$.

Let $\mb f \in \mc H$ and $\lambda \in \C$. Then $(\lambda - L) \mb u = \mb f$ implies 
\begin{align*}
(\lambda + \tfrac{2}{p-1} ) u_1(\rho) + \rho u_1'(\rho) - u_2'(\rho) - p \kappa_0 \smallint_0^\rho u_2(s)ds= f_1(\rho) \\
(\lambda + \tfrac{2}{p-1} ) u_2(\rho) + \rho u_2'(\rho) - u_1'(\rho) = f_2(\rho)
\end{align*}
in a weak sense. Thanks to the boundary condition we obtain from the second equation that
$$u_1(\rho)= \rho u_2(\rho)+ (\lambda  + \tfrac{2}{p-1} -1  )\int_0^\rho u_2(s)ds -\int_0^\rho f_2(s)ds. $$
Inserting this into the first equation yields
\begin{align*} 
-(1-\rho^2) & u_2'(\rho)+2\left(\lambda + \tfrac{2}{p-1} \right) \rho u_2(\rho) + \left(\left(\lambda+ \tfrac{2}{p-1} \right)\left(\lambda + \tfrac{2}{p-1} - 1\right) - p\kappa_0 \right) \smallint_0^{\rho} u_2(s) d s \\ & = f_1(\rho) + \rho f_2(\rho) +
\left(\lambda+ \tfrac{2}{p-1} \right)  \smallint_0^{\rho} f_2(s) d s.
\end{align*}
We set $u(\rho) := \int_0^\rho u_2(s)ds$ and obtain Eq.~ (\ref{Eq:inhom_eigenval}). 
Finally, $u_2 \in L^2(0,1)$ implies $u \in H^1(0,1) \hookrightarrow C[0,1]$ and $\mb u \in \mc D(L)$ yields 
$u \in H^2_\mathrm{loc}(0,1)\cap C^1[0,1)$.
\end{proof}

In order to improve the rough growth estimate given in Corollary \ref{cor:gen}, we analyse 
the spectrum of the generator. 
The next two Lemmas characterize the spectral properties of the generator $L$ sufficiently accurate.

\begin{lemma}
\label{lemma:spect_L}
For the spectrum $\sigma(L)$ of the generator $L$ we have
$$\sigma(L) \subset \left\{\lambda \in \C: Re \lambda \leq \max \{\tilde \omega_p,-1 \} \right\} \cup \{1\}$$
where $\tilde \omega_p=\tfrac12 - \tfrac{2}{p-1}$. 
\end{lemma}

\begin{proof}
Set $M :=\left\{\lambda \in \C: Re \lambda \leq \max \{\tilde \omega_p,-1 \} \right\} \cup \{1\}$.
Let $\lambda \in \sigma(L)$. If $Re \lambda \leq \tilde \omega_p$ then $\lambda \in M$. 
So let us assume that $Re \lambda > \tilde \omega_p$. Then, by Corollary \ref{Cor_specL0}, 
$\lambda \in \sigma(L) \setminus \sigma(L_0)$ and the identity 
$\lambda - L = [1-L'R_{L_0}(\lambda)](\lambda - L_0)$ together with the spectral theorem for compact operators imply that 
$\lambda \in \sigma_p(L)$. Thus, there exists a nontrivial $\mb u \in \mc D(L)$ such that 
$(\lambda - L)\mb u = \mb 0$. By Lemma \ref{Lemma:closure} this implies the existence of a weak solution $u$
of Eq.~ (\ref{Eq:inhom_eigenval}) with right hand side equal to zero. Recall that $u \in H^1(0,1)$ and $u(0)=0$. 
We transform Eq.~ (\ref{Eq:inhom_eigenval}) by substituting $\rho \mapsto z:=\rho^2$ to obtain the hypergeometric 
equation (recall that $\kappa_0=\frac{2(p+1)}{(p-1)^2}$)
\begin{equation}
\label{eq_hypgeom}
z(1-z)v''(z)+[c-(a+b+1)z]v'(z)-abv(z)=0 
\end{equation} 
where $v(z):=u(\sqrt{z})$ and the parameters are given by 
$a= \frac12 (\lambda -2)$, $b=\frac12 (\lambda + \frac{p+3}{p-1})$, $c=\frac12$. 
For $\lambda \neq 1-\frac{2}{p-1}$ a fundamental system  around $z=1$ is given by $\{v_1,\tilde v_1\}$, 
$v_1(z)={}_2F_1(a, b; a+b+1-c;
1-z)$ and 
$\tilde{v}_1(z)=(1-z)^{c-a-b}{}_2F_1(c-a,c-b;c+1-a-b;1-z)$ where ${}_2F_1$ is the standard
hypergeometric function, see e.g., \cite{DLMF}.
The exponent $c-a-b=1-\frac{2}{p-1}-\lambda$ vanishes for $\lambda=1-\frac{2}{p-1}$ and 
in this case one solution is still given by $v_1$ and the second one diverges logarithmically for $z \to 1$. 
Since we assume $Re \lambda > \tfrac12 - \tfrac{2}{p-1}$, $v$ must be a multiple of 
$v_1$ for $u$ to be in $H^1(0,1)$. Around $z=0$ there is a fundamental system
given by $\{v_0,\tilde v_0\}$, where
$v_0(z):=z^{1-c}{}_2F_1(a+1-c,b+1-c; 2-c; z)$ and
$\tilde{v}_0(z):={}_2F_1(a,b;c;z)$. Thus, there exist constants
$c_1$, $c_2$ such that $v_1=c_1 \tilde{v}_0+c_2 v_0$. In order to satisfy 
the boundary condition $v(0)=0$, the coefficient $c_1$, which can be given in terms of the Gamma function
\cite{DLMF}
$$ c_1=\frac{\Gamma(a+b+1-c)\Gamma(1-c)}{\Gamma(a+1-c)\Gamma(b+1-c)},$$
must vanish. Consequently, $c_1=0$ if and only if $a+1-c$ or
$b+1-c$ is a pole, which yields $\frac12(\lambda -1)=-k$ or $\frac{\lambda}{2} + \frac{p+1}{p-1}=-k$
for a $k\in \mathbb{N}_0$. This implies that $\lambda$ is real and
$$\lambda \in \left \{\omega \in \R: \omega > \tilde \omega_p \wedge \left(\omega= 1-2k \vee \omega=-2k-\tfrac{2p+2}{p-1}\right), k=0,1,\dots \right \}.$$  
Since $1<p\leq 3$ we have $\tilde \omega_p \leq - \tfrac12$. If $\tilde \omega_p \geq {-1}$ then 
the only possibility is $\lambda=1$. If $\tilde \omega_p < {-1}$ then either $\lambda=1$ 
or $\lambda \leq -1$. In any case we conclude that $\lambda \in M$.
\end{proof}

\begin{lemma}
\label{Lemma:gauge_eigevalue}
The eigenvalue $1 \in \sigma_p(L)$ has geometric multiplicity equal to one.
The associated geometric eigenspace is spanned by 
\begin{equation}\label{Eq:gaugemode}
\mb{g}(\rho):=\left ( \begin{array}{c}
\frac{p+1}{p-1} \rho  \\ 1 \end{array} \right ).
\end{equation}
In the following $\mb g$ will be referred to as the symmetry mode.
\end{lemma}

\begin{proof}
Note that $\mb g \in \mc D(L)$ and a straightforward calculation yields $(1-L)\mb g=0$. In particular by Lemma \ref{Lemma:closure} and the definition of $\kappa_0$ we infer that 
\begin{equation}
\label{Eq:gauge_eigenvalue_first_comp}
g_1(\rho)=\rho g_2(\rho)+ \tfrac{2}{p-1}\int_0^\rho g_2(s)ds 
\end{equation}
and
\begin{equation}
\label{Eq:gauge_eigenvalue}
-(1-\rho^2) g''(\rho) + \tfrac{2(p+1)}{p-1} \rho g'(\rho) - \tfrac{2(p+1)}{p-1} g(\rho) =0
\end{equation}
for $g(\rho):=\int_0^{\rho} g_2(s)ds=\rho$. 
Suppose there is another eigenfunction $\tilde{\mb g}$ for $\lambda=1$.  Then this corresponds to another (weak) solution $\tilde g(\rho):=\int_0^{\rho} \tilde {g_2}(s)ds$ of Eq.~ (\ref{Eq:gauge_eigenvalue}). 
A fundamental system of Eq.~\eqref{Eq:gauge_eigenvalue} is given by 
$\{h_0,h_1\}$, where $h_0(\rho) = \rho$ and
$$h_1(\rho)=(1-\rho^2)^{-\frac{2}{p-1}} \tilde h_1(\rho)$$
with $\tilde h_1(\rho)={}_2F_1(1,\tfrac12 - \tfrac{p+1}{p-1};\tfrac12;\rho^2)$ and $\tilde h_1(1) \neq 0$ for $1<p\leq 3$.  
However, by Lemma \ref{Lemma:closure}, $\tilde g \in C[0,1]$ and thus it must be a multiple of $h_0=g$. Therefore, there exists a constant $c \in \C$ such that
$$ \int_0^{\rho}\tilde g_2(s)ds =c \int_0^{\rho} g_2(s)ds$$ and we infer that $\tilde g_2=c g_2$. Eq.~ (\ref{Eq:gauge_eigenvalue_first_comp}) implies  $\tilde g_1=c g_1$ and we conclude that $\tilde{\mb g}=c \mb g$.
\end{proof}

\subsection{Spectral projection and linear time evolution restricted to the stable subspace}

The symmetry mode is an explicit example of an exponentially growing solution of the linearized equation. 
However, its origin will only become clear in the course of the nonlinear perturbation theory. 
We will see that it is due to the time translation invariance of the problem and thus, we do not consider 
this instability as a ``physical'' one.
The aim of this section is to remove the symmetry eigenvalue  $\lambda = 1$ via a Riesz projection 
and to obtain a growth estimate for the solution of the linearized equation on the stable subspace. 
We define a projection operator 
\begin{equation}\label{eq:spectral_proj}
P=\frac{1}{2\pi i}\int_\gamma R_L(\lambda) d\lambda,
\end{equation}
where $\gamma$ is a circle that lies entirely in $\rho(L)$ and encloses the eigenvalue $1$ in such a way that no other 
spectral points of $L$ lie inside $\gamma$. The projection $P$ commutes with $L$ in the sense that $PL \subset LP$ and as a consequence, $P$ also commutes with the semigroup generated by 
$L$, i.e., $PS(\tau)=S(\tau)P$ for $\tau \geq 0$. We define subspaces $\mc M=P\mc H$ and $\mc N = (1-P)\mc H$ which decompose 
the operator $L$ into parts living on $\mc M$ and $\mc{N}$, respectively. 
Let $L_{\mc N}$ be defined by $L_{\mc N} \mb u := L \mb u$ with $\mc D(L_{\mc N})=\mc D(L) \cap \mc N$ 
($L_{\mc M}$ is then defined analogously).
Since $\mc N$ and $\mc M$ are closed subspaces we can regard $L_{\mc N}$ and $L_{\mc M}$ as linear operators on the Hilbert spaces 
$\mc N$ and $\mc M$, respectively, with spectra $\sigma(L_{\mc M}) = \{1\}$  and
$\sigma(L_{\mc N}) = \sigma(L)\backslash \{1\}$. In the following we call $\mc M$ the ``unstable subspace''.
The operator $L$ is not self--adjoint and therefore, the next result is nontrivial and 
crucial for the nonlinear perturbation theory.

\begin{lemma}
\label{lem:algebraic_multi}
The unstable subspace $\mc M $ is spanned by the symmetry mode, i.e., $ P \mc H =\langle \mb{g} \rangle$ and the algebraic multiplicity of $1 \in \sigma_p(L)$ is one.
\end{lemma}

\begin{proof}
The case $\dim \mc M = \infty$ can be ruled out by an abstract argument: if $\dim \mc M=\infty$ 
then, by \cite{kato}, p.~239, Theorem 5.28,
$1$ would belong to the essential spectrum of $L$ which is stable under compact perturbations 
(see \cite{kato} p.~244, Theorem 5.35). However, $1 \not \in \sigma(L_0)$ and this yields a contradiction.
We conclude that $L_{\mc M}$ is in fact a finite-dimensional operator. 
Since $1$ is an eigenvalue of $L_{\mc M}$ and, according to Lemma \ref{Lemma:gauge_eigevalue}, 
the corresponding geometric eigenspace is spanned by $\mb g$, we obtain $\mb g \in \mc M$ and thus, 
$\langle \mb{g} \rangle \subset \mc M$.

It remains to prove the reverse inclusion. Note that $(1-L_{\mc M})$ is nilpotent since $0$ is the only eigenvalue, i.e., 
there exists an $m \in \N$ such that $(1-L_{\mc M})^m \mb u =0$ for arbitrary $\mb u \in \mc M$. 
If $m=1$ then $\mc M \subset \ker(1-L_{\mc M})=\langle \mb{g} \rangle$ and we are done. 
Suppose that $m \geq 2$. Then there exists a nontrivial $\mb v \in \mathrm{rg}(1-L_{\mc M})$ such that 
$(1-L_{\mc M}) \mb v = 0$, i.e., $\mb v \in  \ker(1-L_{\mc M})$ and $\mb v$
must therefore be a multiple of the symmetry mode. This shows that there exists a $\mb u \in \mc D(L_{\mc M})$ with $(1-L_{\mc M}) \mb u = c \mb g$. 
We will show that this leads to a contradiction. We set $c=1$ without loss of generality. 
Suppose there exists a function $\mb u$ in $\mc D(L)$ such that $(1-L)\mb u = \mb g$. Then, by Lemma \ref{Lemma:closure}, 
\begin{equation}
-(1-\rho^2) u''(\rho) + \tfrac{2(p+1)}{p-1} \rho u'(\rho) - \tfrac{2(p+1)}{p-1} u(\rho) = g(\rho)
\end{equation}
for $u(\rho):=\int_0^{\rho} u_2(s) ds$ and $g(\rho) := g_1(\rho) + \rho g_2(\rho) +
\tfrac{p+1}{p-1}   \smallint_0^{\rho} g_2(s) d s = \tfrac{3p+1}{p-1} \rho$.
For the homogeneous equation we have the fundamental system $\{h_0,h_1\}$ introduced in the proof of Lemma \ref{Lemma:gauge_eigevalue}, where $h_0(\rho) = \rho$ and $h_1(\rho)=(1-\rho^2)^{-\frac{2}{p-1}} \tilde h_1(\rho)$ with $\tilde h_1$ continuous on $[0,1]$ and $\tilde h_1(0) \neq 0$. The Wronskian is given by 
$$W(h_0,h_1) = -(1-\rho^2)^{-\tfrac{p+1}{p-1}}$$
and thus, a solution of the inhomogeneous equation must be of the form
$$ u(\rho)=c_0 h_0(\rho) +c_1 h_1(\rho)- h_0(\rho) \int_{\rho_0}^\rho h_1(s) g(s) (1-s^2)^{\frac{2}{p-1}} ds
+ h_1(\rho)\int_{\rho_1}^\rho h_0(s) g(s) (1-s^2)^{\frac{2}{p-1}}  ds $$
for some constants $c_0,c_1 \in \C$ and $\rho_0,\rho_1 \in [0,1]$. 
The boundary condition at $\rho=0$ implies 
$c_1= - \int_{\rho_1}^0 h_0(s) g(s) (1-s^2)^{\frac{2}{p-1}}  ds$ and inserting the definitions of $h_0, h_1$ and $g$ yields 
$$ u(\rho)=c_0 \rho - \tfrac{3p+1}{p-1} \rho \int_{\rho_0}^\rho s \tilde h_1(s)ds
+ \tfrac{3p+1}{p-1} (1-\rho^2)^{-\frac{2}{p-1}} \tilde h_1(\rho) \int_{0}^\rho s^2(1-s^2)^{\frac{2}{p-1}} ds.$$
Since $u$ belongs to $C[0,1]$ (Lemma \ref{Lemma:closure}), we must have $\int_{0}^1 s^2(1-s^2)^{\frac{2}{p-1}} ds=0$.
However, this is impossible since the integrand is strictly positive for all $s \in (0,1)$.
\end{proof}

In order to improve the growth estimate in Lemma \ref{cor:gen} we apply a well-known theorem by Gearhart, Pr\"{u}ss and Greiner. To this end we need the following result, which states that the resolvent is uniformly bounded in some right half plane.
In the following we set $H_a:=\{\lambda \in \C: \mathrm{Re} \lambda \geq a\}$ for $a \in \R$. 

\begin{lemma}\label{ResolventBounds}
For any $\varepsilon > 0$ there exist constants $c_1,c_2 > 0$ such that $$\|R_{L}(\lambda)\| \leq c_1$$
for all $\lambda \in H_{\tilde \omega_p + \varepsilon}$ with $|\lambda| \geq c_2$.
 \end{lemma}

\begin{proof}
Fix $\varepsilon > 0$ and let $\lambda \in H_{\tilde \omega_p + \varepsilon}$ where $\lambda \not \in \{1 - \frac{2}{p-1}, 1\}$. 
We use the identity 
$$R_L(\lambda)=R_{L_0}(\lambda)[1-L'R_{L_0}(\lambda)]^{-1}$$ to obtain uniform bounds on the resolvent for $|\lambda|$ large.
By definition of $L'$ we have
$$L'R_{L_0}(\lambda) \mb f =  \left ( \begin{array}{c} p\kappa_0 K[R_{L_0}(\lambda)\mb f]_2
\\ 0 \end{array} \right )$$
where $K: L^2(0,1) \to L^2(0,1)$ is defined by $Ku(\rho) = \int_0^\rho u(s) ds$.
For $\mb f \in \mc H$ consider the equation $(\lambda - L_0)\mb u = \mb f$. 
Its solution is given by $\mb u=R_{L_0}(\lambda)\mb f$.  Lemma \ref{Lemma:closure} yields
$$[R_{L_0}(\lambda)\mb f]_1(\rho) = (\lambda -1 + \tfrac{2}{p-1}) K[R_{L_0}(\lambda)\mb f]_2(\rho) +
\rho [R_{L_0}(\lambda)\mb f]_2(\rho) - Kf_2(\rho).$$
The estimate in Lemma \ref{Cor_specL0} implies
$$\|[R_{L_0}(\lambda)\mb f]_j\|_{L^2(0,1)} \leq \|R_{L_0}(\lambda)\mb f \| \leq \frac{\|\mb f\|}{|\mathrm{Re} 
\lambda - \tilde \omega_p|}$$
for $j =1,2$ and we obtain
$$ \|K[R_{L_0}(\lambda)\mb f]_2 \|_{L^2(0,1)} \lesssim \frac{\|\mb f\|}{|\lambda -1 + \tfrac{2}{p-1}|}.$$
Thus, for $|\lambda|$ sufficiently large, the Neumann series
$$ [1-L'R_{L_0}(\lambda)]^{-1} =  \sum_{k=0}^{\infty} [L'R_{L_0}(\lambda)]^k$$
converges and the claim follows. 
\end{proof}

We conclude the linear perturbation theory with an estimate of the linear evolution on the stable subspace.

\begin{proposition}\label{Prop:LinearTimeEvolution}
Let $P$ be the spectral projection defined in Eq.~(\ref{eq:spectral_proj}) and
set \[ \omega_p := \max \left \{-1,\tfrac{1}{2} - \tfrac{2}{p-1} \right\}. \] 
Then, for any $\varepsilon>0$, there exists a constant $C_{\varepsilon} > 0$ such that the 
semigroup $S(\tau)$ given in Corollary \ref{cor:gen} satisfies
\begin{equation}\label{Est:LinearTimeEvolution}
\|S(\tau)(1-P)\mb f\| \leq C_{\varepsilon} e^{(-|\omega_p|+\varepsilon) \tau}\|(1-P)\mb f\|
\end{equation}
for all $\tau \geq 0$ and $\mb f \in \mc H$.
Furthermore, $S(\tau)P \mb f = e^{\tau} P \mb f$.
\end{proposition}

\begin{proof}
The operator $L_{\mc N}$ is the generator of the subspace semigroup $S(\tau)\rst{_\mc N}$ and its resolvent is given
by $R_L(\lambda)\rst{_\mc N}$. The first estimate follows from the uniform boundedness of the resolvent in $H_{\omega_p+\varepsilon}$ (Lemma \ref{ResolventBounds}) and the theorem by Gearhart, Pr\"{u}ss and Greiner (see for example \cite{engel}, p. 302, Theorem 1.11).
The second assertion follows from $P \mc H = \langle \mb g \rangle$ and the fact that $\mb g$ is an eigenfunction of the linear operator $L$ with eigenvalue $1$.
\end{proof}

\section{Nonlinear perturbation theory}

\subsection{Preliminaries}
Now we turn to the full nonlinear problem. The following two lemmas will be used frequently.
\begin{lemma}
\label{est:Linfty}
If $u \in L^2(0,1)$ then $\tilde{u}$, defined by
$\tilde u(\rho):=\frac{1}{\sqrt \rho} \int_0^{\rho} u(s) ds$, belongs to $L^{\infty}(0,1)$ and satisfies
\[ \|\tilde{u}\|_{L^\infty(0,1)}\leq \|u\|_{L^2(0,1)}. \]
\end{lemma}
\begin{proof}
First note that $\rho \mapsto \int_0^{\rho} u(s) ds$ is a continuous function on $[0,1]$ for $u \in L^2(0,1)$.
Using the Cauchy-Schwarz inequality we esimate
$$\left |\tilde u(\rho) \right| = \left |\frac{1}{\sqrt \rho}  \smallint_0^{\rho} u(s) ds\right|  \leq \|u\|_{L^2(0,1)}$$
for $\rho \in (0,1]$.
Taking the essential supremum yields the claim.
\end{proof}
We will also use Hardy's inequality in the following form.
\begin{lemma}\label{est:Hardy}
For $u\in L^2(0,1)$ we have
$$ \int_0^1 \frac{|\int_0^{\rho} u(s) ds|^2}{\rho^2}d\rho \lesssim \int_0^1 |u(\rho)|^2d\rho. $$
\end{lemma}

\subsection{Estimates for the nonlinearity}
From now on we restrict ourselves to real--valued functions. We introduce a function $n: \R \times [0,1] \to \R$ defined by
$$n(x,\rho):=\rho \bigg( |\kappa_0^{\frac{1}{p-1}}+x|^{p-1} (\kappa_0^{\frac{1}{p-1}}+x) - p\kappa_0 x - \kappa_0^{\frac{p}{p-1}} \bigg),$$
cf.~Eq.~\eqref{eq:N}.
It is easy to see that  
\[
 |n(x,\rho)| \lesssim \left\{
  \begin{array}{l l}
    \rho |x|^2 & \quad |x| <  1\\
    \rho |x|^p & \quad |x| \geq 1.\\
  \end{array} \right.
\]
A convenient way to write this is $|n(x,\rho)|\lesssim \rho |x|^2\langle x\rangle^{p-2}$ with
the ``japanese bracket'' $\langle x \rangle:=\sqrt{1+|x|^2}$.
In the following we denote by $B_1$ and $\mc B_1$ the open unit balls in $L^2(0,1)$ and $\mc H$, respectively. 
To (formally) define the nonlinearity we introduce an operator $A: L^2(0,1) \to L^2(0,1)$,
$$Au(\rho) := \frac{1}{\rho} \int_0^{\rho} u(s) ds.$$
An application of Hardy's inequality shows that $A$ is bounded. We set
$$N(u)(\rho) := n(Au(\rho),\rho).$$

\begin{lemma}\label{lemma:nonlinL2}
The operator $N$ maps $L^2(0,1)$ into $L^2(0,1)$. Furthermore, there exist constants $c_1,c_2 > 0$ such that for $u,v \in B_1$ 
$$\|N(u) \|_{L^2} \leq c_1 \|u \|_{L^2}^2 $$
and
$$\|N(u)-N(v)\|_{L^2} \leq c_2 (\| u \|_{L^2} + \| v \|_{L^2}) \|u-v\|_{L^2}.$$
\end{lemma}

\begin{proof}
Note that for $1 < p \leq 3$ the function $n$ defined as above is at least once continuously differentiable with respect to 
$x$ and we have the bound
\begin{equation*}
|\partial_1 n(x,\rho)| \lesssim \rho |x|\langle x \rangle^{p-2}
\end{equation*}
for all $x \in \R$ and $\rho \in [0,1]$ which, in particular, implies $\partial_1 n(0,\rho)=0$ and hence,
$N(0)=0$.
By the fundamental theorem of calculus we infer that for $x,y \in \R$
\begin{eqnarray*}
|n(x,\rho)-n(y,\rho)|  &\leq & |x-y| \int_0^1 |\partial_1 n(y +h(x-y),\rho)| dh \\ 
& \lesssim & \rho |x-y| \int_0^1  |y +h(x-y)|\langle y+h(x-y)\rangle^{p-2}dh  \\
& \lesssim & \rho |x-y| \left \{ \begin{array}{c} |x|+|y|\quad p \in (1,2] \\
|x|\langle x\rangle^{p-2}+|y|\langle y \rangle^{p-2}\quad p \in (2,3] \end{array} \right . . 
\end{eqnarray*}

Now we prove the estimate for the nonlinear operator $N$.
The following argument works only for $1< p \leq 3$, since for higher exponents the singular factors at $\rho=0$ can no longer be controlled.
For $u,v \in L^2(0,1)$ we write $\tilde u(\rho) := \int_0^{\rho} u(s) ds$ and $\tilde v(\rho) := \int_0^{\rho} v(s) ds$. 
We distinguish two cases. If $p \in (1,2]$ we readily estimate
\begin{align*}
\|N(u)-N(v)\|^2_{L^2} &= \int_0^1 | n(Au(\rho),\rho)- n(Av(\rho),\rho)|^2 d\rho \\
& \lesssim \int_0^1 \rho^2 |Au(\rho)-Av(\rho)|^2 (|Au(\rho)|^2 +|Av(\rho)|^2)  d\rho \\
&\lesssim (\|\tilde{u}\|_{L^\infty}^2+\|\tilde{v}\|_{L^\infty}^2)\int_0^1 |Au(\rho)+Av(\rho)|^2 d\rho \\
&\lesssim (\|u\|_{L^2}^2+\|v\|_{L^2}^2)\|u-v\|_{L^2}^2
\end{align*}
by Lemma \ref{est:Linfty} and Hardy's inequality. 
On the other hand, if $p\in (2,3]$, we proceed similarly and obtain
\begin{align*}
\|N(u)-N(v)\|^2_{L^2} 
& \lesssim \int_0^1 \rho^2 |Au(\rho)-Av(\rho)|^2 (|Au(\rho)|^2\langle Au(\rho)\rangle^{2(p-2)} +|Av(\rho)|^2
\langle Av(\rho)\rangle^{2(p-2)})  d\rho \\
&\lesssim \int_0^1 \rho^{3-p} |Au(\rho)-Av(\rho)|^2 \\
&\quad \times \left (|\rho^{-\frac12}\tilde{u}(\rho)|^2
\langle \rho^{-\frac12}\tilde{u}(\rho)\rangle^{2(p-2)}
+|\rho^{-\frac12}\tilde{v}(\rho)|^2
\langle \rho^{-\frac12}\tilde{v}(\rho)\rangle^{2(p-2)}\right)  d\rho \\
&\lesssim \left (\|u\|_{L^2}^2\langle \|u\|_{L^2}\rangle^{2(p-2)}+\|v\|_{L^2}^2\langle \|v\|_{L^2} \rangle^{2(p-2)} \right )
\|u-v\|_{L^2}^2
\end{align*}
again by Lemma \ref{est:Linfty} and Hardy's inequality.
Since $N(0)=0$ we immediately conclude the boundedness of $N$ on $L^2(0,1)$.
In particular, we have $\|N(u)\|_{L^2} \lesssim \|u\|^2_{L^2}$ for $u \in B_1$. For $u,v \in B_1$ the above estimates yield 
$$\|N(u)-N(v)\|_{L^2} \lesssim  (\| u \|_{L^2} + \| v \|_{L^2} )\| u -v \|_{L^2}$$
as claimed.
\end{proof}
Finally for $\mb u =(u_1,u_2)^T \in \mc H$ we define the vector valued nonlinearity by
$$\mb {N}(\mb u) :=  \left ( \begin{array}{c} N(u_2)  \\ 0  \end{array}\right).  $$

\begin{lemma}\label{est:nonlinearity}
The nonlinearity $\mb {N}$ maps $\mc H$ into $\mc H$, $\mb {N}(\mb 0)=\mb 0$ and there exist constants $c_1,c_2 > 0$ such that for $\mb u$,$\mb v \in \mc B_1$
$$\| \mb {N}(\mb u) \| \leq c_1 \|\mb u \|^2$$
and
$$\| \mb {N}(\mb u) - \mb {N}(\mb v) \| \leq c_2 (\|\mb u\| + \|\mb v\|) \|\mb u - \mb v\|.$$
Furthermore, $\mb {N}$ is Fr\'{e}chet differentiable at $\mb 0$ and $D\mb {N}(\mb 0)=\mb 0$.
\end{lemma}

\begin{proof}
For $\mb u,\mb v \in \mc B_1$ we apply the result of Lemma \ref{lemma:nonlinL2} to obtain
\begin{eqnarray*}
\| \mb {N}(\mb u) - \mb {N}(\mb v) \|^2 = \| N(u_2) - N(v_2) \|^2_{L^2} \lesssim (\| u_2 \|^2_{L^2} + \| v_2 \|^2_{L^2}) \|u_2-v_2\|^2_{L^2}  \lesssim (\| \mb u\|^2 + \| \mb v \|^2) \|\mb u-\mb v\|^2.
\end{eqnarray*}
This implies
\begin{align*}
 & \| \mb {N}(\mb u) - \mb {N}(\mb v) \| \lesssim (\| \mb u\|^2 + \| \mb v \|^2)^{\frac{1}{2}} \|\mb u-\mb v\| \lesssim (\| \mb u\| + \| \mb v \|) \|\mb u-\mb v\|.
\end{align*}
We have $\mb {N}(\mb 0)=\mb 0$ which implies $\| \mb {N}(\mb v) \| \lesssim \|\mb v \|^2$. In particular, there exists a constant $c$ 
independent of $\mb v$ such that 
$$\frac{\| \mb {N}(\mb v) \|}{\|\mb v\|} \leq c \|\mb v \|.$$
Since the left hand side vanishes in the limit $\mb v \to \mb 0$, we infer that $\mb N$ 
is Fr\'echet differentiable at zero with $D\mb {N}(\mb 0)\mb =\mb 0$. 
\end{proof}

\subsection{Abstract formulation of the nonlinear equation}
We turn to the full nonlinear problem and write Eq.~ (\ref{eq:nonlinear_firstorder_css}) as an ordinary differential equation on $\mc H$.
With the nonlinearity defined as above it reads 

\begin{equation}\label{eq:nonlinear_ODE}
\left \{ \begin{array}{l}
\frac{d}{d\tau}\Phi(\tau)=L\Phi(\tau)  + \mb {N}(\Phi(\tau)) \mbox{ for }\tau>-\log T \\
\Phi(-\log T)=\mb{u}
\end{array} \right .
\end{equation}
for $\Phi: [-\log T,\infty) \to \mc H$ and initial data $\mb u \in \mc H$. We rewrite the above system as an integral equation, 

\begin{equation*}
\Phi(\tau) = S(\tau + \log T) \mb u + \int_{- \log T}^{\tau} S(\tau - \tau') \mb {N}(\Phi(\tau')) d\tau' \quad \text{for} \quad \tau \geq - \log T,
\end{equation*}
i.e., we are looking for mild solutions of Eq.~ (\ref{eq:nonlinear_ODE}). 
In order to remove the dependence of the equation on the blow up time $T$ we introduce 
a new variable $\Psi:[0,\infty) \to \mc H$ defined by
$$\Psi(\tau) := \Phi(\tau - \log T)$$
such that the above integral equation is now equivalent to
\begin{equation}\label{eq:Psi_integral}
\Psi(\tau) = S(\tau) \mb u + \int_{0}^{\tau} S(\tau - \tau') \mb {N}(\Psi(\tau')) d\tau' \quad \text{for} \quad \tau \geq 0.
\end{equation}
We study this equation on a Banach space $\mc X$ defined as 
$$ \mc X := \left \{ \Psi \in  C([0,\infty),\mc H): \sup_{\tau > 0} e^{\mu_p \tau} \|\Psi(\tau)\| < \infty \right \}$$
with norm 
$$ \|\Psi\|_{\mc X} := \sup_{\tau > 0} e^{\mu_p \tau} \|\Psi(\tau)\|$$
where
$$\mu_p := |\omega_p| - \varepsilon = \min\{1,\tfrac{2}{p-1}- \tfrac12 \} -\varepsilon,$$
cf.~Proposition \ref{Prop:LinearTimeEvolution}, where $\varepsilon > 0$ is arbitrary but fixed
and without loss of generality we assume $\varepsilon$ so small that $\mu_p > 0$.  
In the following, estimate \eqref{Est:LinearTimeEvolution} will be used frequently, hence most of the constants will depend on $\varepsilon$. 
However, for notational convenience we will only indicate this dependence in the proof of the main result.

\subsection{Global existence for corrected (small) initial data}

We follow the strategy of \cite{wavemapsnonlin}. First we study the following equation,
\begin{equation}\label{eq:codim-one_integraleq}
\Psi(\tau) =\ S(\tau) (1-P) \mb u - \int_{0}^{\infty} e^{\tau - \tau'} P \mb {N}(\Psi(\tau')) d\tau'
 +  \int_{0}^{\tau} S(\tau - \tau') \mb {N}(\Psi(\tau')) d\tau'. 
\end{equation}
This is the original equation modified by a correction term in order to suppress the instability coming from the symmetry mode. 
We use a fixed point argument to show existence of solutions of Eq.~ (\ref{eq:codim-one_integraleq}).
In a further step we account for the time translation symmetry of the problem and show that 
the correction can be annihilated by adjusting the blow up time $T$ (which is now encoded in the initial data)
such that we end up with a solution of Eq.~ (\ref{eq:Psi_integral}). 
For  $\delta > 0$ we define $\mc X_{\delta} \subset \mc X$  by
$$\mc X_\delta := \{ \Psi \in \mc X: \|\Psi  \|_{\mc X} \leq \delta\}.$$
\begin{lemma}\label{lemma:nonlinear_delta_est}
For $0<\delta < 1$ there exist constants $c_1,c_2 >0$ such that 
$$\|\mb {N}(\Psi(\tau)) \| \leq c_1\delta^2 e^{-2 \mu_p \tau}$$
and
$$\|\mb {N}(\Psi(\tau)) - \mb {N}(\Phi(\tau)) \| \leq c_2 \delta e^{- \mu_p \tau} \|\Psi(\tau) - \Phi(\tau) \|$$
for $\Phi, \Psi \in \mc X_{\delta}$ and $\tau >0$.
\end{lemma}

\begin{proof}
Let $\Psi \in \mc X_{\delta}$. Then $ \|\Psi(\tau) \| \leq \delta e^{-\mu_p\tau} < 1$ for all $\tau >0$ and $\delta < 1$.
Lemma \ref{est:nonlinearity} implies that there exists a constant $c_1 >0$ such that 
$$\|\mb {N}(\Psi(\tau)) \| \leq c_1  \| \Psi(\tau)\|^2  \leq c_1 \delta^2 e^{-2 \mu_p \tau}.$$
Let $\Phi \in \mc X_{\delta}$. Then there exists a constant $c_2 >0$ such that
\begin{align*}
\|\mb {N}(\Psi(\tau)) - \mb {N}(\Phi(\tau)) \| \leq \frac{c_2}{2} (\|\Psi(\tau)\| + \| \Phi(\tau) \|)  \|\Psi(\tau) - \Phi(\tau) \| \leq  c_2 \delta e^{- \mu_p \tau} \|\Psi(\tau) - \Phi(\tau) \|
\end{align*}
which implies the second estimate.
\end{proof}
We abbreviate the right hand side of Eq.~ (\ref{eq:codim-one_integraleq}) by defining the operator
\begin{equation}\label{def:K}
\mb K(\Psi,\mb u)(\tau):= S(\tau) (1-P) \mb u - \int_{0}^{\infty} e^{\tau - \tau'} P \mb {N}(\Psi(\tau')) d\tau'
 +  \int_{0}^{\tau} S(\tau - \tau') \mb {N}(\Psi(\tau')) d\tau'.
 \end{equation}

\begin{lemma}\label{lemma:Kselfandcontract}
For  $\delta>0$ sufficiently small and fixed $\mb u \in \mc H$, with $\|\mb u\| \leq \delta^2$, the operator $\mb K$ 
maps $\mc X_\delta$ into itself and is contracting, in particular
$$\|\mb K(\Phi,\mb u) - \mb K(\Psi, \mb u) \|_{\mc X} \leq \frac {1}{2} \|\Phi - \Psi\|_{\mc X}$$
for $\Phi, \Psi \in \mc X_{\delta}$.
\end{lemma}

\begin{proof}
For fixed $(\Psi,\mb u)$ with $\Psi \in \mc X_\delta$ and $\mb u \in \mc H$ the integrals occuring in the operator $\mb K$ can be viewed as Riemann integrals
over continuous functions, which exist since  $\| P \mb {N}(\Psi(\tau)) \| \lesssim 1$ by Lemma \ref{lemma:nonlinear_delta_est}. 
To see that $\mb K(\Psi, \mb u) \in \mc X_\delta$ for $\delta$ small enough we decompose the operator 
according to
$$\mb K(\Psi, \mb u) = P\mb K(\Psi, \mb u)+(1-P)\mb K(\Psi, \mb u).$$
We apply the results of Proposition \ref{Prop:LinearTimeEvolution} and Lemma \ref{lemma:nonlinear_delta_est}.
Let $\|\mb u\| \leq \delta^2$. Then for $\tau \geq 0$ we obtain
\begin{eqnarray*}
\|P\mb K(\Psi, \mb u)(\tau)\| & = & \left \|- \int_{0}^{\infty} e^{\tau - \tau'} P \mb {N}(\Psi(\tau')) d\tau'
 +  \int_{0}^{\tau} S(\tau - \tau') P\mb {N}(\Psi(\tau')) d\tau' \right \| \\
 & \leq & \int_{\tau}^{\infty} e^{\tau - \tau'}\| P \mb {N}(\Psi(\tau'))\| d\tau' \leq c_1 \delta^2 \int_{\tau}^{\infty} e^{\tau-\tau' (1 + 2\mu_p)} d\tau' \lesssim \delta ^2 e^{-2\mu_p \tau},
\end{eqnarray*}
and
\begin{eqnarray*}
\|(1-P) \mb K(\Psi, \mb u)(\tau)\| & \leq & \|S(\tau)(1- P) \mb u \| +   \int_{0}^{\tau}\left  \| S(\tau - \tau')(1-P) \mb {N}(\Psi(\tau')) \right \| d\tau'    \\
& \lesssim  & e^{-\mu_p \tau} \|\mb u\| + \int_{0}^{\tau} e^{-\mu_p(\tau - \tau')} \|\mb {N}(\Psi(\tau'))\|d\tau'
 \lesssim  \delta^2 e^{-\mu_p \tau} + \delta^2\int_{0}^{\tau} e^{-\mu_p(\tau + \tau')} d\tau' \\
& \lesssim  &  \delta^2 e^{-\mu_p\tau}.
\end{eqnarray*}
We infer that there exist constants $c_1,c_2 > 0$ such that
$$\|P\mb K(\Psi, \mb u)(\tau)\|  \leq c_1 \delta^2 e^{-\mu_p \tau},$$
and
$$\|(1-P) \mb K(\Psi, \mb u)(\tau)\| \leq c_2 \delta^2 e^{-\mu_p \tau}.$$
Thus for $\delta \leq \min \{1, \frac{1}{2c_1}, \frac{1}{2c_2} \}$ we obtain
$$\|\mb K(\Psi, \mb u)(\tau) \| \leq \|P \mb K(\Psi, \mb u)(\tau)\| + \|(1-P)\mb K(\Psi, \mb u)(\tau) \| 
\leq \frac{\delta}{2} e^{-\mu_p \tau} + \frac{\delta}{2}e^{-\mu_p \tau} \leq \delta e^{-\mu_p \tau}.$$
Continuity of $\mb K(\Psi, \mb u)(\tau)$ as a function of $\tau$ follows essentially from strong continuity of the semigroup 
(cf.~Lemma 3.10 in \cite{wavemapsnonlin}).  
It is left to show that $\mb K$ is contracting. Let $\Psi, \Phi \in \mc X_{\delta}$. Then
\begin{eqnarray*}
& & \|P\mb K(\Phi, \mb u)(\tau)- P\mb K(\Psi, \mb u)(\tau)\| \leq   \int_{\tau}^{\infty} e^{\tau - \tau'} 
\| P \mb {N}(\Phi(\tau')) -P \mb {N}(\Psi(\tau')) \| d\tau' \\
& \lesssim & \delta  \int_{\tau}^{\infty} e^{\tau-\tau'(1+\mu_p)} \|\Phi(\tau')- \Psi(\tau') \| d\tau' 
 \lesssim  \delta \sup_{\sigma > \tau} e^{\mu_p \sigma}  \|\Phi(\sigma)- \Psi(\sigma) \| \int_{\tau}^{\infty} 
e^{\tau-\tau' (1 + 2\mu_p)} d \tau' \\
& \lesssim & \delta  e^{-2 \mu_p \tau} \|\Phi - \Psi\|_{\mc X}. 
\end{eqnarray*}
Similarly,
\begin{eqnarray*}
& & \|(1-P)\mb K(\Phi, \mb u)(\tau)- (1-P)\mb K(\Psi, \mb u)(\tau)\| \leq 
\int_{0}^{\tau} \| S(\tau - \tau')(1-P)(\mb {N}(\Phi(\tau')) - \mb {N}(\Psi(\tau'))) \|d\tau'   \\
& \lesssim & \int_{0}^{\tau}  e^{-\mu_p (\tau - \tau')} \|\mb {N}(\Phi(\tau')) - \mb {N}(\Psi(\tau')) \| d\tau' \lesssim
\delta \int_{0}^{\tau} e^{-\mu_p \tau} \|\Phi(\tau')- \Psi(\tau') \| d\tau' \\
& \lesssim &
\delta \sup_{\sigma \in (0,\tau)} e^{\mu_p \sigma}  \|\Phi(\sigma)- \Psi(\sigma) \| \int_{0}^{\tau} e^{-\mu_p (\tau + \tau')} d\tau' \lesssim  \delta e^{-\mu_p\tau} \|\Phi - \Psi\|_{\mc X}.
\end{eqnarray*}
This shows that for $\delta$ sufficiently small, 
$$\sup_{\tau > 0} e^{\mu_p\tau} \|P\mb K(\Phi, \mb u)(\tau)- P\mb K(\Psi, \mb u)(\tau)\| \leq \frac{1}{4}  \|\Phi - \Psi\|_{\mc X},$$
and 
$$\sup_{\tau > 0} e^{\mu_p\tau} \|(1-P)\mb K(\Phi, \mb u)(\tau)- (1-P)\mb K(\Psi, \mb u)(\tau)\| \leq \frac{1}{4}  \|\Phi - \Psi\|_{\mc X}, $$
which implies the claim.
\end{proof}

\begin{theorem}\label{Th:ExistenceModEq}
For $\mb u \in \mc B_1 \subset \mc H$ sufficiently small, there exists a unique solution $\Psi(\cdot;\mb u) \in \mc X$ of 
\begin{equation}\label{eq:Kmodified_integral}
\Psi(\cdot;\mb u)=\mb K(\Psi(\cdot;\mb u),\mb u).
\end{equation}
Moreover, the map $\bs \Psi: \mc U  \subset \mc B_1\to \mc X$ defined by $\mb \Psi (\mb u) = \Psi(\cdot;\mb u)$ is continuous 
 and Fr\'{e}chet differentiable at $\mb u = \mb 0$ where $\mc U$ denotes a sufficiently small open neighbourhood of zero in $\mc H$.
\end{theorem}

\begin{proof}
Lemma \ref{lemma:Kselfandcontract} and the fact that  $\mc X_{\delta}$ is a closed subset yield a unique fixed point of Eq. \eqref{eq:Kmodified_integral}
in $\mc X_{\delta}$. That this is indeed the unique solution in the whole space $\mc X$ follows by standard arguments 
(see also the proof of Theorem \ref{Th:Global_existence_shifted_equation}).
Note that for $\mb u, \mb v \in \mc U$ we have $\mb \Psi(\mb u),\mb \Psi(\mb v) \in \mc X_{\delta}$ and
\begin{eqnarray*}
& & \|\mb \Psi(\mb u) -\mb \Psi(\mb v)\|_{\mc X} = \|\mb K(\mb \Psi(\mb u),\mb u) -\mb K(\mb \Psi(\mb v),\mb v) \|_{\mc X} \\
& \leq & 
 \|\mb K(\mb \Psi(\mb u),\mb u) -\mb K(\mb \Psi(\mb v),\mb u) \|_{\mc X} +  \|\mb K(\mb \Psi(\mb v),\mb u) -\mb K(\mb \Psi(\mb v),\mb v) \|_{\mc X}.
\end{eqnarray*}
By Lemma \ref{lemma:Kselfandcontract},
$$ \|\mb K(\mb \Psi(\mb u),\mb u) -\mb K(\mb \Psi(\mb v),\mb u) \|_{\mc X} \leq \frac{1}{2} \|\mb\Psi(\mb u) -\mb \Psi(\mb v) \|_{\mc X}.$$
Inserting the definition of $\mb K$ yields
$$\|\mb K(\mb \Psi(\mb v),\mb u)(\tau) -\mb K(\mb \Psi(\mb v),\mb v)(\tau) \| = \|S(\tau)(1-P)(\mb u-\mb v)\| \leq e^{-\mu_p\tau} \|\mb u-\mb v\|$$
and we conclude that
\begin{equation}\label{Solmap_lipschitz}
 \|\mb \Psi(\mb u) - \mb \Psi(\mb v)\|_{\mc X} \lesssim \|\mb u-\mb v\|,
\end{equation}
which implies continuity. We claim that the solution map $\mb \Psi$ is Fr\'echet differentiable at $\mb u = \mb 0$. We define an
auxiliary operator $\tilde D\mb  \Psi(\mb 0):\mc H \to \mc X$ by $[\tilde D \mb \Psi(\mb 0)\mb v](\tau):=S(\tau)(1-P)\mb v$ for $\mb v \in \mc H$. 
It is obvious that this defines a bounded linear operator from $\mc H$ into $\mc X$. We show that it is indeed the Fr\'echet derivative, i.e., 
$$\lim_{\mb v \to \mb 0} \frac{1}{\| \mb v\|}\|\mb \Psi(\mb v) - \mb \Psi(\mb 0) - \tilde D \mb \Psi(\mb 0)\mb v\|_{\mc X} = 0.$$
Recall that $\mb N(\mb 0) = \mb 0$, hence $\mb \Psi(\mb 0) = \mb 0$ is a solution of Eq.~\eqref{eq:Kmodified_integral} for $\mb u=\mb 0$. 
We assume that $\mb v \in \mc U$, such that $\mb \Psi(\mb v) = \mb K(\mb \Psi(\mb v),\mb v)$. 
Inserting the definition of $\mb K$ we compute
\begin{align*}
\mb \Psi(\mb v)(\tau) - S(\tau)(1-P)\mb v  &= \int_0^\tau S(\tau - \tau') \mb N(\mb  \Psi(\mb v)(\tau'))d\tau' -
\int_0^{\infty} e^{\tau -\tau'} P \mb N(\mb  \Psi(\mb v)(\tau'))d\tau' \\
&=: G(\mb \Psi(\mb v))(\tau).
\end{align*}
Again we write $ G(\mb  \Psi(\mb v))(\tau)= P[G(\mb  \Psi(\mb v))(\tau)]+(1-P)[G(\mb  \Psi(\mb v))(\tau)]$. Estimate (\ref{Solmap_lipschitz}) and 
calculations similar to those in the proof of Lemma \ref{lemma:Kselfandcontract} yield 
\begin{eqnarray*}
\|P[G(\mb \Psi(\mb v))(\tau)] \| & \leq & \int_\tau^\infty e^{\tau -\tau'} \|P \mb N(\mb \Psi(\mb v)(\tau'))\| d\tau'
 \leq  \int_\tau^\infty e^{\tau -\tau'} \|\mb \Psi(\mb v)(\tau')\|^2  d\tau' \\
& \lesssim & \|\mb v\|^2 \int_\tau^\infty e^{\tau -\tau'(1+2\mu_p)} d\tau' \lesssim \|\mb v\|^2 e^{-2\mu_p \tau}.
\end{eqnarray*}
Similarly,
\begin{eqnarray*}
\|(1-P)[G(\mb \Psi(\mb v))(\tau)] \| & \leq & \int_0^{\tau} \|S(\tau - \tau')(1-P)\mb N(\mb  \Psi(\mb v)(\tau'))\| d\tau'  \\
&\lesssim & \int_0^{\tau}  e^{-\mu_p(\tau -\tau')} \|\mb N(\mb  \Psi(\mb v)(\tau'))\| d\tau' \lesssim   \|\mb v\|^2 e^{-\mu_p \tau}.
\end{eqnarray*}
We infer that $\|G(\mb  \Psi(\mb v))\|_{\mc X} \lesssim \|\mb v\|^2$ and thus $\lim_{\mb v \to \mb 0}  \frac{1}{\| \mb v\|}\|G(\mb  \Psi(\mb v))\|_{\mc X}=0,$ 
which implies the claim.
\end{proof}

\subsection{Global existence for arbitrary (small) initial data}

The aim of this section is to use the existence result of Theorem \ref{Th:ExistenceModEq} to obtain 
a solution of the original wave equation for arbitrary initial data (close to $\psi^T$). 
Up to now we implicitly assumed the blow up time $T$ to be fixed. However, arbitrary
 perturbations of the initial data will change the blow up time and we account for this fact by allowing $T$ to vary. 
Recall that the initial data we want to prescribe are of the form
\begin{equation}\label{origin_initialdata}
 \Psi(0)(\rho) =\left ( \begin{array}{c} \rho T^{\frac{p+1}{p-1}}g(T\rho)-\tfrac{2 \rho}{p-1}\kappa_0^{\frac{1}{p-1}} 
\\ T^{\frac{2}{p-1}}\left (T \rho f'(T\rho)+f(T\rho) \right ) -  \kappa_0^{\frac{1}{p-1}} \end{array} \right ),
\end{equation}
see Eq.~\eqref{eq:nonlinear_firstorder_css}.
We separate the dependence on $T$ and the free data $(f,g)$ by introducing
\begin{equation}\label{Def:initialdata_v_kappa}
\mb v(\rho):= \left ( \begin{array}{c}   \rho g(\rho) - \frac{2 \rho}{p-1} \kappa_0^{\frac{1}{p-1}} \\
		        \rho f'(\rho)+f(\rho)  -   \kappa_0^{\frac{1}{p-1}}          \end{array} \right ), \quad 
			 \mb \kappa(\rho):= \kappa_0^{\frac{1}{p-1}} \left ( \begin{array}{c} \frac{2 \rho}{p-1}  \\ 1 \end{array} \right ),
\end{equation}									
which are the initial data relative to the fundamental self--similar solution for $T=1$. 
We rewrite the right hand side of (\ref{origin_initialdata}) and define 
$$\mb U(\mb v,T)(\rho):=T^{\frac{2}{p-1}} [\mb v(T \rho) + \kappa(T\rho)] - \kappa(\rho).$$
The data have to be prescribed on the interval $[0,T]$ and we are confronted with the problem that we do not know $T$ in advance. 
As in \cite{wavemapsnonlin} the argument will be perturbative around $T=1$ and therefore it suffices to restrict $T$ to 
the interval $I=(\frac12,\frac32)$. In the following we set 
$$\tilde{\mc H}:= L^2(0,\tfrac32) \times L^2(0,\tfrac32).$$
\begin{lemma}\label{initialdata_op}
The function $\mb U: \tilde{\mc H} \times I \to \mc H$ is continuous and $\mb U(\mb 0,1) = \mb 0$. 
Furthermore $\mb U(\mb 0,\cdot):  I \to \mc H$ is Fr\'echet differentiable and
 $$[D_T \mb U(\mb 0, T)\rst{_{T=1}} \lambda](\rho) = \tfrac{2\lambda}{p-1} \kappa_0^{\frac{1}{p-1}} \mb g(\rho),$$
where $\lambda \in \R$ and $\mb g$ denotes the symmetry mode (cf. Eq.~ (\ref{Eq:gaugemode})).
\end{lemma}

\begin{proof}
The proof of continuity is similar to the proof of Lemma $3.14$ in \cite{wavemapsnonlin}. 
We define $J:L^2(0,\frac32) \times I \to L^2(0,1)$ by   $J(v,T)(\rho):=v(T\rho)$. 
For fixed $T$ the map $J(\cdot,T): L^2(0,\frac32) \to L^2(0,1)$ is Lipschitz-continuous since 
$$ \|J(v,T) - J(\tilde v,T)\|_{L^2(0,1)}^2 = \int_0^1 |v(T\rho) - \tilde v(T \rho)|^2 d\rho = \frac{1}{T} \int_0^T  |v(\rho) - \tilde v(\rho)|^2 d\rho 
 \leq 2 \| v - \tilde v\|_{L^2(0,\frac32)}^2$$
and the continuity is uniform with respect to $T$. It is therefore sufficient to show that for fixed $v \in L^2(0,\frac32)$ the function 
$J(v,\cdot): I \to L^2(0,1)$ is continuous. This can be seen by noting that for all $v, \tilde v \in L^2(0,\frac32)$ and $T, \tilde T \in I$
\begin{align*}
 \| J(v,T) - J(v,\tilde T)\|_{L^2(0,1)}  &\leq  \|J(v,T) - J(\tilde v,T)\|_{L^2(0,1)} + \|J(\tilde v,T) -J(\tilde v,\tilde T)\| _{L^2(0,1)}   \\
& \quad +\|J(\tilde v,\tilde T) - J(v,\tilde T)\|_{L^2(0,1)} \\
&\lesssim \| v- \tilde v\|_{L^2(0,\frac32)} + \|J(\tilde v,T) - J(\tilde v,\tilde T)\|_{L^2(0,1)}.
\end{align*}
Thus, for any given $\epsilon > 0$ we can find a $\tilde v \in C[0,\frac32]$ such that
\begin{equation}
 \| J(v,T) - J(v,\tilde T)\|_{L^2(0,1)} <\frac{\epsilon}{2} + c\left(\int_0^1 |\tilde v (T\rho) - \tilde v(\tilde T\rho)|^2 d\rho 
 \right)^{\frac12}
\end{equation}
for some constant $c>0$ since $C[0,\frac32]$ is dense in $L^2(0,\frac32)$.
By the continuity of $\tilde v$, the integral vanishes in the limit $T \to \tilde T$.
The above results imply continuity of $J:L^2(0,\frac32) \times I \to L^2(0,1)$ and thus,
$$\mb U(\mb v,T)=  \left ( \begin{array}{c} T^{\frac{2}{p-1}}( J(v_1,T)+ J(\kappa_1,T) ) - 
\kappa_1 \\ T^{\frac{2}{p-1}}( J(v_2,T)+ J(\kappa_2,T) ) - \kappa_2 \end{array} \right )$$
is continuous for $\mb v=(v_1,v_2)^T \in \tilde{\mc H}$ and 
$\kappa = (\kappa_1,\kappa_2)^T$ as defined in Eq.~ (\ref{Def:initialdata_v_kappa}). 

To show differentiability we set $\mb v = \mb 0$ and consider $\mb U(\mb 0,\cdot): I \to \mc H$, which is given by
$$\mb U(\mb 0,T)(\rho) = T^{\frac{2}{p-1}}\kappa(T\rho) - \kappa(\rho) =  \kappa_0^{\frac{1}{p-1}} 
\left ( \begin{array}{c} \frac{2 \rho}{p-1} \left (T^{\frac{p+1}{p-1}} - 1 \right) \\ T^{\frac{2}{p-1}} - 1 \end{array} \right )$$
The map is obviously differentiable for all $T \in I$. Recalling the definition of the symmetry mode in Lemma \ref{lemma:spect_L} we obtain
$$[D_T \mb U(\mb 0, T)\rst{_{T=1}} \lambda](\rho) = \frac{2\lambda}{p-1}\kappa_0^{\frac{1}{p-1}} 
 \left ( \begin{array}{c} \frac{(p+1)}{p-1} \rho \\ 1 \end{array} \right ) = \frac{2\lambda}{p-1}\kappa_0^{\frac{1}{p-1}}  \mb g(\rho)$$
for $\lambda \in \R$.
\end{proof}

With these technical results at hand we now turn to the original problem. In the previous section we showed existence of solutions 
for the modified integral equation (\ref{eq:codim-one_integraleq}) with initial data $\mb u \in \mc U$, where $\mc U$ denotes a 
sufficiently small neighbourhood of $\mb 0 \in \mc H$. We rewrite the initial data in terms of $T$ and $\mb v$ as defined in
 Eq.~ (\ref{Def:initialdata_v_kappa}). Inserting in the definition yields $\mb U(\mb 0,1)=\mb 0$. By continuity $\mb U(\mb v,T) \in \mc U$ 
provided that $(\mb v, T) \in \mc V \times \tilde I$ where $\mc V$ and $\tilde I$ are sufficiently small neighbourhoods 
of $\mb 0 \in \tilde {\mc H}$ and $1 \in I$, respectively.
By Theorem \ref{Th:ExistenceModEq} there exists a solution $\mb U(\mb v,T) \mapsto  \mb \Psi( \mb U(\mb v,T)) \in \mc X$.
Recall that Eq.~ (\ref{eq:codim-one_integraleq}) is Eq.~ (\ref{eq:Psi_integral}) modified by an expontential factor times the function  
$\mb F: \mc V \times \tilde I \to \langle \mb g \rangle$ defined by
$$\mb F(\mb v, T):= P \left (\mb U(\mb v,T) + \int_{0}^{\infty} e^{- \tau'} \mb {N}(\mb \Psi(\mb U(\mb v,T))(\tau')) d\tau' \right).$$
Evaluation yields $\mb F(\mb 0, 1)=\mb 0$, i.e., for $\mb v = \mb 0$ and $T=1$ the correction vanishes and $\mb \Psi( \mb U(\mb 0,1)) = \mb 0$ is 
also a solution of Eq.~ (\ref{eq:Psi_integral}). In the following we show that for every small $\mb v$ there exists a $T$ close to one, such that this still holds true.
We need the next lemma as a prerequisite.

\begin{lemma}\label{cont_diff_correction}
$\mb F: \mc V \times \tilde I \subset \tilde H \times I \to \langle \mb g \rangle$  is continuous. Moreover
$\mb F(\mb 0,\cdot): \tilde I \to  \langle \mb g \rangle$ is Fr\'echet differentiable at $T=1$ and 
$$D_T \mb F(\mb 0,T)\rst{_{T=1}} \lambda = \tfrac{2\lambda}{p-1} \kappa_0^{\frac{1}{p-1}} \mb g$$
for $\lambda \in \R$.
\end{lemma}

\begin{proof}
To rewrite $\mb F$ in a more abstract way we introduce the integral operator $\mb B:\mc X \to \mc H, \Psi \mapsto \int_0^{\infty}e^{-\tau'} \Psi(\tau') d\tau'$, 
which is linear and bounded since
$$\|\mb B \Psi \| \leq \int_0^{\infty} e^{-\tau'} \|\Psi(\tau')\| d\tau'  \leq \sup_{\tau' >0} \|\Psi(\tau')\| \leq \|\Psi\|_{\mc X}.$$ 
We define $ \mb {\tilde N}: \mc X \to \mc X$ by $\mb {\tilde N}(\Psi)(\tau):=\mb {N}(\Psi(\tau))$. We claim that $\mb {\tilde N}$ 
is Fr\'echet differentiable at $\mb 0 \in \mc X$ and the Fr\'echet derivative at zero is given by $D \mb {\tilde N}(\mb 0) \Psi = \mb 0$ for $\Psi \in \mc X$.
This follow from $\mb {\tilde N}(\mb 0)=\mb 0$ and
$$\| \mb {\tilde N}(\Psi) \|_{\mc X} = \sup_{\tau > 0} e^{\mu_p\tau} \|\mb {N}(\Psi(\tau))\| \lesssim \sup_{\tau > 0} e^{\mu_p\tau} 
\|\Psi(\tau) \|^2 \lesssim \| \Psi \|^2_{\mc X} \quad \text{for} \quad \Psi \in \mc X_{\delta}.$$
Thus $$\frac{\| \mb {\tilde N}(\Psi) \|_{\mc X} }{ \|\Psi \|_{\mc X} } \lesssim \|\Psi \|_{\mc X} $$ 
with a constant independent of $\Psi$, which implies the claim. Now 
$$\mb F(\mb v, T) = P\left[ \mb U(\mb v,T) + \mb B \mb {\tilde N}(\mb \Psi(\mb U(\mb v,T))) \right].$$
By Lemma \ref{initialdata_op} and the continuity of $\mb {\tilde N}$ and $\mb \Psi$, respectively, we see that $\mb F$ is continuous.
To show differentiability we set $\mb v = \mb 0$ and obtain
$$\mb F(\mb 0, T) = P\left[ \mb U(\mb 0,T) + \mb B \mb {\tilde N}(\mb \Psi(\mb U(\mb 0,T))) \right].$$
The right hand side is differentiable at $T=1$ by Theorem \ref{Th:ExistenceModEq}, Lemma \ref{initialdata_op} and the above considerations. 
We conclude that
\begin{eqnarray*}
D_T \mb F(\mb 0,T)\rst{_{T=1}} \lambda & = & P D_T \mb U(\mb 0,T)\rst{_{T=1}}  \lambda + 
P \mb B  D \mb {\tilde N}(\mb 0) D \mb \Psi(\mb 0) D_T \mb U(\mb 0,T)\rst{_{T=1}} \lambda  \\
& = & P D_T \mb U(\mb 0,T)\rst{_{T=1}} \lambda =  \tfrac{2\lambda}{p-1} \kappa_0^{\frac{1}{p-1}} \mb g.
\end{eqnarray*}
\end{proof}

\begin{lemma}
Let $\tilde {\mc V} \subset \tilde H$ be a sufficiently small neighbourhood of $\mb 0$. For every $\mb v \in \tilde {\mc V}$ there 
exists a $T \in \tilde I \subset (\frac12,\frac32)$, such that $\mb F(\mb v, T) = \mb 0$.
\end{lemma} 

\begin{proof}
The range of $\mb F$ is contained in $\langle \mb g \rangle$, which is a one dimensional vector space. Thus, there 
exists an isomorphism $i: \langle \mb g \rangle \to \R $ such that
  $i(c \mb g)=c $ for $c \in \R$. We set $f := i \circ \mb F$, where $f: \mc V \times \tilde I \to \R$ is continuous
and $\mb F(\mb 0, 1)=\mb 0$ implies $f(\mb 0, 1)=0$. Lemma \ref{cont_diff_correction} shows that $f(\mb 0,\cdot): \tilde I \to \R$ is differentiable
at $T=1$ and $D_T f(\mb 0,T)\rst{_{T=1}} \neq 0$.
Consequently, there exist values $T_1,T_2 \in \tilde I$ such that $f(\mb 0,T_1) > 0$ and $f(\mb 0, T_2)< 0$. 
Continuity of $f$ with respect to the first variable implies that there exists an open neighbourhood $\tilde {\mc V} \subset \mc V$ such
 that $f(\mb v, T_1) >0$ and $f(\mb v, T_2) < 0$ for $\mb v \in \tilde {\mc V}$. For $\mb v \in \tilde  {\mc V}$ consider $f(\mb v, \cdot): \tilde I \to \R$. 
By continuity of $f(\mb v,T)$ with respect to $T$ and the intermediate value theorem we conclude that
there exists a $T^* \in (T_1,T_2)$ such that $f(\mb v, T^*)=0$.
\end{proof}

This yields the next result.
\begin{theorem}
\label{Th:Global_existence_shifted_equation}
Let $\mb v \in \tilde H$ be sufficiently small. 
Then there exists a $T$ close to $1$ such that
\begin{equation}
\label{Eq:IntEqPsi}
\Psi(\tau)=S(\tau) \mb U(\mb v,T) + \int_{0}^{\tau} S(\tau - \tau') \mb {N}(\Psi(\tau'))d\tau', \quad \tau \geq 0
\end{equation}
has a continuous solution $\Psi: [0, \infty) \to \mc H$ satisfying
$$\| \Psi (\tau)\| \leq \delta e^{-\mu_p \tau}$$
for all $\tau \geq 0$ and some $\delta \in (0,1)$. Moreover, this solution is unique in $C([0,\infty),\mc H)$.
\end{theorem}

\begin{proof}
The existence of a solution $\Psi \in \mc X_\delta$ follows from the above considerations. 
Let $\Phi \in C([0,\infty),\mc H)$ be another solution satisfying the same equation. We assume that $\Psi \neq \Phi$.
By continuity, there exists an $\varepsilon \in (0,\tfrac{1-\delta}{2})$ and a $\tau_0 > 0$ such that
\begin{equation*}
\varepsilon < \|\Psi(\tau_0)-\Phi(\tau_0) \|
\end{equation*}
and 
$$\|\Psi(\tau)-\Phi(\tau) \| < 2 \varepsilon, \quad \tau \in [0,\tau_0],$$
which yields $\| \Phi(\tau) \| < 1$. 
For $\tau \in [0,\tau_0]$ we obtain
\begin{eqnarray*}
\|\Psi(\tau)-\Phi(\tau) \|  & \leq & c \int_0^{\tau} e^{\tau-\tau'} \|\mb {N}(\Psi(\tau')) -\mb {N}(\Phi(\tau')) \|d\tau' \\
& \leq & C(\tau_0) (e^{\tau} -1) \sup_{\tau' \in [0,\tau]} \|\Psi(\tau')-\Phi(\tau') \| \\
\end{eqnarray*}
by applying Lemma \ref{est:nonlinearity}. We infer that there exists a $\tau_1 \in (0,\tau_0]$ such that 
$$\sup_{\tau \in [0,\tau_1]} \|\Psi(\tau)-\Phi(\tau) \|  \leq \frac12 \sup_{\tau \in [0,\tau_1]} \|\Psi(\tau)-\Phi(\tau) \| $$ 
which implies $\Psi(\tau)=\Phi(\tau)$ for all $\tau \in [0,\tau_1]$. Iterating this argument yields $\Psi(\tau)=\Phi(\tau)$ for $\tau \in [0,\tau_0]$, 
which contradicts $\|\Psi(\tau_0)-\Phi(\tau_0) \| > \varepsilon $.
\end{proof}

\begin{proposition}(Global existence for arbitrary, small initial data)
\label{Prop:Global_existence}
Let $\varepsilon > 0$ be small enough such that $\mu_p =|\omega_p| - \varepsilon > 0$. Let $\mb v \in L^2(0,\frac32) \times L^2(0,\frac32)$ be sufficiently small. Then there exists a $T$ close to $1$ such that
\begin{equation}\label{Eq:IntegralEqPhi}
\Phi(\tau)=S(\tau + \log T) \mb U(\mb v,T) + \int_{-\log T}^{\tau} S(\tau - \tau') \mb {N}(\Phi(\tau'))d\tau', \quad \tau \geq - \log T
\end{equation}
has a continuous solution $\Phi: [-\log T, \infty) \to \mc H$ satisfying
$$\| \Phi (\tau)\| \leq C_{\varepsilon} e^{-\mu_p \tau}$$
for all $\tau \geq - \log T$ and a constant $C_{\varepsilon} > 0$ depending on $\varepsilon$. 
Moreover, this solution is unique in $C([-\log T, \infty) ,\mc H)$. 
Thus, $\Phi$ is the unique global mild solution of Eq.~ (\ref{eq:nonlinear_ODE}) with initial data $\Phi(-\log T)=\mb U(\mb v,T)$.
\end{proposition}

\subsection{Proof of Theorem \ref{Th:Main}}

\begin{proof}
We translate the result of Proposition \ref{Prop:Global_existence} back to the original coordinates $(t,r)$. 
Let $(f,g)$ satisfy the assumption of Theorem \ref{Th:Main}. 
For the fundamental self--similar solution with $T=1$ we have 
$$\psi^1(0,r)=\kappa_0^{\frac{1}{p-1}}, \quad \psi_t^1(0,r)=\tfrac{2}{p-1} \kappa_0^{\frac{1}{p-1}}.$$ We define
$$ v_1(\rho) := \rho g(\rho) - \tfrac{2 \rho}{p-1} \kappa_0^{\frac{1}{p-1}}, \quad  v_2(\rho) := f(\rho) + \rho f'(\rho) - \kappa_0^{\frac{1}{p-1}},$$
such that $\mb v =(v_1,v_2)^T \in L^2(0,\frac32) \times L^2(0,\frac32)$ and 
$$\|\mb v \|_{\tilde {\mc H}} = \|(f,g) - (\psi^1(0,\cdot),\psi_t^1(0,\cdot))\|_{\mc E(\frac{3}{2})}$$
We may assume $\mb v$ small enough to satisfy the assumptions of Proposition \ref{Prop:Global_existence}
and we infer that there exists a unique global mild solution $\Phi \in C([-\log T,\infty),\mc H)$ of Eq.~ (\ref{eq:nonlinear_ODE}) for $T$ close
to $1$ with initial data $\Phi(-\log T) =\mb U(\mb v,T)$ and
$$\| \Phi(\tau) \| \leq C_{\varepsilon} e^{-(|\omega_p|-\varepsilon)\tau}$$
for all $\tau \geq - \log T$.
By definition $$\Phi(\tau)(\rho) = (\phi_1(\tau,\rho),\phi_2(\tau,\rho))^T$$ is a solution of Eq.~ (\ref{eq:nonlinear_firstorder_css}) and 
Eq.~  (\ref{Eq:ReconstructField}) yields
$$\psi(t,r) = \psi^T(t,r) + (T-t)^{-\frac{2}{p-1}} r^{-1} \int_0^r \phi_2(-\log(T-t),\tfrac{r'}{T-t}) dr'$$
and
$$\psi_t(t,r)= \psi_t^T(t,r) + (T-t)^{-\frac{2}{p-1}} r^{-1} \phi_1(-\log(T-t),\tfrac{r}{T-t}).$$
For $\varphi = \psi - \psi^T$ we obtain
\begin{small}
\begin{eqnarray*}
\|(\varphi(t,\cdot),\varphi_t(t,\cdot) \|_{\mc E(T-t)}^2 & = & (T-t)^{-\frac{4}{p-1}} \left( \int_0^{T-t} |\phi_2(-\log(T-t),\tfrac{r}{T-t})|^2 dr + 
 \int_0^{T-t} |\phi_1(-\log(T-t),\tfrac{r}{T-t})|^2 dr  \right) \\
& =&  (T-t)^{\frac{p-5}{p-1}} \left( \int_0^{1} |\phi_2(-\log(T-t),\rho)|^2 d\rho +  \int_0^{1} |\phi_1(-\log(T-t),\rho)|^2 d\rho  \right)\\
& = & (T-t)^{\frac{p-5}{p-1}} \|\Phi(-\log(T-t))\|^2 \leq C^2_{ \varepsilon} (T-t)^{\frac{p-5}{p-1} + 2(|\omega_p|-\varepsilon)}.
\end{eqnarray*}
\end{small}
Thus, 
$$\|(\psi(t,\cdot),\psi_t(t,\cdot)) - (\psi^T(t,\cdot),\psi^T_t(t,\cdot) )\|_{\mc E(T-t)} \leq C_{\varepsilon} (T-t)^{\frac{p-5}{2(p-1)} 
+|\omega_p|- \varepsilon}.$$
\end{proof}

\bibliography{wave}
\bibliographystyle{plain}

\end{document}